\documentclass[10pt]{amsart}
\usepackage [utf8]{inputenc}
\usepackage {bm}
\usepackage {fontenc}
\usepackage {amsfonts}
\usepackage {amssymb}
\usepackage {amsmath}
\usepackage {amsthm}\usepackage {mathtools}
\usepackage {enumerate}
\usepackage {enumitem}
\usepackage [pagebackref,colorlinks,linkcolor=blue,citecolor=blue,urlcolor=blue,hypertexnames=true]{hyperref}
\usepackage [pagebackref,hypertexnames=true]{hyperref}
\usepackage {mathrsfs}
\usepackage {marginnote}
\usepackage {xcolor}
\usepackage {soul}\usepackage {tikz}\usetikzlibrary {calc}
\usepackage [all,cmtip]{xy} \usepackage {caption}

\usepackage{comment}
\usepackage[normalem]{ulem}

\numberwithin{equation}{section}
\newtheorem{theorem}{Theorem}

\newtheorem{cor}[theorem]{Corollary}
\newtheorem{lemma}[theorem]{Lemma}
\newtheorem{prop}[theorem]{Proposition}
\theoremstyle{definition}
\newtheorem{definition}[theorem]{Definition}
\newtheorem{claim}[theorem]{Claim}
\newtheorem{assumption*}[theorem]{Assumption}
\newtheorem{problem}[theorem]{Problem}
\theoremstyle{remark}
\newtheorem{remark}[theorem]{Remark}
\newtheorem*{acknowledgements*}{Acknowledgements}

\numberwithin{theorem}{section}

\newcommand{\eps}{\varepsilon}

\newcommand{\N}{\mathbb{N}}
\newcommand{\R}{\mathbb{R}}

\definecolor{darkgreen}{rgb}{0.0, 0.6, 0.0}

\begin{document}

\title[Homeomorphisms between the spheres of $\ell_\infty^k$ and $\ell_1^k$]{On the uniform continuity of homeomorphisms between the spheres of $\ell_\infty^k$ and $\ell_1^k$ }%

\date{\today} 

\author[Braga]{Bruno M. Braga}
\address[B. M. Braga]{IMPA, Estrada Dona Castorina 110, 22460-320, Rio de Janeiro, Brazil}
\email{demendoncabraga@gmail.com}
\urladdr{\url{https://sites.google.com/site/demendoncabraga}}
\thanks {B. M. Braga  was partially supported by FAPERJ, grant E-26/200.167/2023,  by CNPq, grant 303571/2022-5, and by Serrapilheira, grant R-2501-51476.}

\author[Gartland]{C. Gartland}
\address[C. Garland]{Department of Mathematics and Statistics, University of
North Carolina at Charlotte, Charlotte, Noth Carolina, USA}
\email{cgartla1@charlotte.edu}
\urladdr{\url{https://chrisgartland.wordpress.com/}}
\thanks{C. Gartland was partially supported by NSF Award DMS-2546184}

\author[Lancien]{G. Lancien}
\address[G. Lancien]{Universit\'e Marie et Louis Pasteur, CNRS, LmB (UMR 6623), F-25000 Besan\c con, France.}
\email{gilles.lancien@univ-fcomte.fr}
\thanks{G. Lancien was partially supported by the French ANR project No. ANR-24-CE40-0892-01.}

\author[Motakis]{P. Motakis}
\address[P. Motakis]{Department of Mathematics and Statistics, York University, 4700 Keele Street, Toronto, Ontario, M3J 1P3, Canada}
\email{pmotakis@yorku.ca}
\urladdr{\url{https://pmotakis.mathstats.yorku.ca/}}
\thanks{P. Motakis was partially supported by NSERC Grant RGPIN-2021-03639.}

\author[Perneck\'a]{E. Perneck\'a}
\address[E. Perneck\'a]{Faculty of Information Technology, Czech Technical University in Prague, Th\'akurova 9, 160 00, Prague 6, Czech Republic}
\email{eva.pernecka@fit.cvut.cz}
\thanks{}

\author[Schlumprecht]{Th. Schlumprecht}
\address[Th. Schlumprecht]{Texas A\&M University, College Station, TX 77843,
USA and Faculty of Electrical Engineering, Czech Technical University in Prague, Zikova 4, 166 27, Prague}
\email{t-schlumprecht@tamu.edu}
\urladdr{\url{https://people.tamu.edu/~t-schlumprecht/}}
\thanks{Th. Schlumprecht was partially supported by NSF Award DMS-2349322}

\begin{abstract}
We consider the problem of whether there is a sequence of homeomorphisms $(F_k)_k$ between the unit spheres of the $k$-dimensional Banach spaces $\ell_\infty^k$ and $\ell_1^k$ which is also equi-uniformly continuous. We prove that this cannot be the case if the sequence $(F_k)_k$ either (1) {\it does not increase support sizes} (which is a property strictly weaker than support preservation) or (2) is {\it step preserving} (which is a property strictly weaker than being equivariant with respect to permutations of the canonical basis). We also provide quantitative estimates relating the moduli of uniform continuity of the maps to the dimension of the spaces. This gives partial answers to a question of W. B. Johnson and it is related to the problem of whether $c_0$ has Kasparov and Yu's Property (H). Our results also apply to more general spaces other than $\ell_1$ such as spaces with unconditional bases which are not equivalent to the standard $c_0$ basis. Finally, we derive an asymptotic concentration inequality that must be satisfied by step preserving equi-uniformly continuous maps defined on the positive parts of these unit spheres.
\end{abstract}
  
\maketitle

 \section{Introduction}
In 2012,  Kasparov and  Yu introduced a new and intriguing property for Banach spaces: {\it{Property (H)}}. Recall, denoting the unit sphere of a Banach space $Y$ by $S_Y$, we say that a Banach space $X$ has {\it{Property (H)}} if there are a uniformly continuous map $F\colon S_X\to S_{\ell_2}$ and increasing sequences $(X_n)_n$ and $(H_n)_n$ of finite dimensional  subspaces such that $\bigcup_nX_n$ and $\bigcup_n H_n$ are dense in $X$ and $\ell_2$, respectively, and such that the restrictions of $F$ to each $ S_{X_n}$ is a homeomorphism onto $S_{H_n}$ (see \cite[Definition 1.1]{KasparovYu2012GeoTop}). We call the sequence $(X_n)_n$ a {\it paving} of $X$. Property (H) is quite unique in a few aspects. While it may look at first glance to be a sort of ``nonlinear embedding/equivalence property'' between Banach spaces, this is really not the case since the inverses of the restrictions $F\restriction S_{X_n}$ need not be equi-uniformly continuous. This lack of uniform continuity places Property (H) in an unusual spot: it is somewhat simultaneously  a geometric and a topological property which, on top of it, is also related to the local theory of Banach spaces since it must preserve the pavings $(X_n)_n$ and $(H_n)_n$.
 
Due to its relation to the Novikov conjecture ---  one of the most important unsolved problems in topology --- it is hard to overstate the importance of Property (H). Kasparov and Yu showed in \cite[Theorem 1.2]{KasparovYu2012GeoTop} that  if a countable discrete group $\Gamma$ admits a coarse embedding into a Banach space with Property (H), then the Novikov conjecture holds for $\Gamma$. In particular, if there is a Banach space with Property (H) which is coarsely universal for all countable discrete groups, then the Novikov conjecture holds unconditionally. Since by a famous result of Aharoni  (\cite[Theorem in Page 288]{Aharoni1974Israel}) every separable metric space Lipschitzly embeds into $c_0$ --- the Banach space of sequences converging to zero --- Kasparov and Yu asked whether $c_0$ has Property (H) (see \cite[Page 1861]{KasparovYu2012GeoTop}).

\begin{problem}[G.\ Kasparov and G.\ Yu]
    Does $c_0$ have Property (H)? In fact, do all separable Banach spaces have Property (H)?\label{ProblemPropH}
\end{problem}

Given the magnitude that a positive answer to Problem \ref{ProblemPropH} would have, it seems sensible to expect the problem to have a negative answer. While, on the one hand, it is very surprising that no examples of separable Banach spaces failing Property (H) are known, on the other hand, there is currently  only one method to show Property (H) holds. In summary: Property (H) remains very poorly understood to date.  

The known method of showing that certain spaces have Property (H) relies on   Mazur maps. Recall, if $p\in [1,\infty)$, then the map 
\begin{align*}
    M_{p}\colon S_{\ell_p}&\to S_{\ell_2}\\
(x_n)_n &\mapsto\left(\mathrm{sign}(x_n)|x_n|^{p/2}\right)_n
\end{align*}
is a uniform homeomorphism between the unit spheres of $\ell_p$ and $\ell_2$ (\cite[Theorem 9.1]{BenyaminiLindenstrauss2000Book}). Considering the finite dimensional spaces $(\ell_p^k)_k$ and $(\ell_2^k)_k$ as subspaces of $\ell_p$ and $\ell_2$, respectively,   these maps  restrict to uniform homeomorphisms between the unit spheres of $\ell_p^k$ and $\ell_2^k$. Hence all $\ell_p$'s, for $p\in [1,\infty)$, have Property (H). 
More generally, using generalized  Mazur maps, one can obtain that every Banach space $X$ with an unconditional basis which does not contain  the spaces $(\ell_\infty^k)_k$ with bounded distortion also has Property (H) (see \cite[Theorem 1]{ChengWang2018JMAA}). This is so because the unit sphere of any such space $X$ is uniformly homeomorphic to the one of $\ell_2$ by a ``Mazur type map'' which preserves the canonical pavings given by   the unconditional basis of $X$ and by  the canonical basis of $\ell_2$ (see \cite[Theorem 2.1]{OdellSchlumprecht1994Acta} and its proof). Moreover, Property (H) also holds for Banach lattices not containing $(\ell_\infty^k)_k$ and for Schatten $p$-classes, for $p\in [1,\infty)$, as noticed in \cite[Page 1860]{KasparovYu2012GeoTop}.

To help illustrate how little is known  about Problem \ref{ProblemPropH}, we mention that it is not even known if $c_0$ fails the requirements of Property (H) with its standard paving. Precisely, is there a uniformly continuous map $F\colon S_{c_0} \to S_{\ell_2}$ such that, for some increasing sequence $(n_k)_k$ of naturals, the restriction of $F$ to the unit sphere of each $\ell_\infty^{n_k}$ has the whole $S_{\ell_2^{n_k}}$ as its homeomorphic image? As asked by Johnson in \cite{JohnsonOverFlow}, even the following remains strikingly open:

\begin{problem}[W.\ B.\ Johnson]
Is there a sequence of homeomorphisms $(F_k)_k$ between the unit spheres of $\ell_\infty^k$ and $\ell_2^k$ whose Lipschitz constants are (uniformly) bounded?\label{ProblemJohnson}
\end{problem}

A positive answer to Problem \ref{ProblemJohnson} would also imply that the Novikov conjecture holds unconditionally. Indeed, if this problem has a positive answer then one can easily show that the  $\ell_2$-sum of $(\ell_\infty^{k})_k$ has Property (H) and, by a result of F. Baudier and G. Lancien, this space is Lipschitzly universal for all locally finite metric spaces (see \cite[Theorem 2.1]{BaudierLancien2008PAMS}); in particular, it contains all countable discrete groups coarsely.

It is worth mentioning that it is well known that the unit spheres of $\ell_\infty^k$ and $\ell_2^k$ cannot be Lipschitzly equivalent by maps with uniformly bounded distortion, i.e., maps $(F_k)_k$ such that $\sup_k\mathrm{Lip}(F_k)\mathrm{Lip}(F_k^{-1})<\infty$. Indeed, if this were the case, then Rademacher theorem would imply that the $(\ell_{\infty}^k)_k$ isomorphically embed with bounded distortion into $\ell_2$, which is not the case. \\

We shall now explain the main findings of this paper. In a nutshell,  this paper   provides the first partial results towards a negative answer for Problem \ref{ProblemJohnson}, which can also be seen as further evidence of the  plausibility of a negative answer for  Problem \ref{ProblemPropH}. 
Before delving into details, a comment is in place about our approach: on the one hand, we are not able to give a definitive negative answer for Problem \ref{ProblemJohnson} at the moment; as the reader will see below, our results  require some extra conditions on the maps $(F_k)_k$. On the other hand, our conclusions about the maps $(F_k)_k$ will be much stronger since we rule out not only the existence of a uniform bound for their Lipschitz constants, but we actually show the functions cannot be equi-uniformly continuous. Furthermore,   we also provide quantitative estimates for the dimensions of the spaces. Consequently, this ties Problems \ref{ProblemPropH} and \ref{ProblemJohnson} even closer together.

The next two theorems give a strong negative answer to Problem \ref{ProblemJohnson} for certain classes of maps. We start by properly defining these classes. For that, given $k\in \N$ and $x=(x_i)_{i=1}^k\in \R^k$, $\mathrm{supp}(x)$ denotes $\{i\in \{1,\ldots,k\}\mid x_i\neq 0\}$. Given $F\colon S_{\ell_\infty^k}\to S_{\ell_1^k}$, we write $F =(F_i)_{i=1}^k$; so, each $F_i$ is the composition of $F$ with the canonical projection of $\ell_1^k$ onto its $i$th coordinate.  
\begin{enumerate}[label=\Roman*]
    \item We say that a map $F\colon S_{\ell^k_\infty}\to S_{\ell_1^k}$   {\it{does not increase support size}} if  
  the cardinality of $\mathrm{supp}(F(x))$ is at most the one of  $\mathrm{supp}(x)$ for all  $ x\in S_{\ell_\infty^k}.$
    \item  We say that $F$ is  {\it{step preserving}}   if for all $x=(x_i)_{i=1}^k\in S_{\ell_\infty^k}$
and all $i,j\in \{1,\ldots, k\}$,  $x_i=x_j$ implies $F_i(x)=F_j(x)$. 
\end{enumerate}
We postpone until Section \ref{SectionExamples} a discussion of examples of maps   satisfying  these properties. For now,  we just point out the obvious:  (1) the class of maps which do not increase support size includes all {\it support preserving maps}, i.e.,  maps such that $\mathrm{supp}(F(x))=\mathrm{supp}(x)$ for all $x\in S_{\ell_\infty^k}$, and (2) the class of  step preserving maps includes all maps which are equivariant with respect to basis permutation (see more details in Section \ref{SectionExamples} below).
  
We obtain the following theorem for functions which do not increase support size. Recall that, for a map $F\colon (X,d_X)\to (Y,d_Y)$ between metric spaces, its {\it modulus of uniform continuity} is given by 
\[\omega_F(t)=\sup\left\{d_Y(F(x),F(y))\mid x,y\in X, \: d_X(x,y)\leq t\right\}, \text{ for }\ t\in [0,\infty).\]
We say that a sequence of maps $(F_k\colon X_k\to Y_k)_k$ between metric spaces is {\it equi-uniformly continuous} if
\[\forall\eps>0, \ \exists \delta>0\text{ such that } \omega_{F_k}(\delta)<\eps,\ \text{for all}\  k\in\N.\]

\begin{theorem}\label{Thm.Intro.Supp.Card.Preserv}
Let $d \in \N$. Then for any $k \geq 19^{2d-1}+1$ and any homeomorphism $F:S_{\ell_\infty^k}\to S_{\ell_1^k}$ that does not increase support size, we have $\omega_F(\frac1d)\ge \frac12$. In particular, there is no sequence $(F_k:S_{\ell_\infty^k}\to S_{\ell_1^k})_{k=1}^\infty$ of equi-uniformly continuous homeomorphisms that do not increase support size.
\end{theorem}
 
The proof of Theorem \ref{Thm.Intro.Supp.Card.Preserv} passes first through obtaining the same conclusion for support preserving maps. In this case, the map $F$ does not even need to be continuous: we show that no sequence of support preserving maps from the spheres of $\ell_\infty^k$ to the ones of $\ell_1^k$ can be equi-uniformly continuous (see Theorem \ref{Thm.Main}). Theorem \ref{Thm.Intro.Supp.Card.Preserv} is then obtained by showing that, if the map $F$ is furthermore assumed to be a homeomorphism, then, up to a permutation operator, it is in fact support preserving (see Lemma \ref{LemmaCardOfSupportPreservingToSupportPreserving}). For both these results, it will be necessary to apply an averaging procedure which makes the map  equivariant with respect to basis permutation and, in particular, also step preserving (see Section \ref{SectionExamples} and Lemma \ref{Lemma.Equivariant.Injective} for details). For this, Brouwer's invariance of domain theorem and its consequence that continuous injective maps between spheres of $k$-dimensional spaces is surjective will be crucial. 
 
We obtain the following theorem for  step preserving functions.   

\begin{theorem}\label{Thm.Intro.Homeo.Equiv.Intro}
Let $d \in \N$. Then for any  $k \geq 19^{2d-1}+1$ and any step preserving continuous map   $F:S_{\ell_\infty^k}\to S_{\ell_1^k}$ such that $F(1,\ldots,1)\neq F(-1,\ldots,-1)$, we have $\omega_F(\frac1d)\ge \frac12$. In particular, there is no sequence $(F_k:S_{\ell_\infty^k}\to S_{\ell_1^k})_{k=1}^\infty$ of equi-uniformly continuous step preserving homeomorphisms.
\end{theorem}

Notice that, while Theorems \ref{Thm.Intro.Supp.Card.Preserv} and \ref{Thm.Intro.Homeo.Equiv.Intro} regard maps from the unit spheres of $\ell_\infty^k$ onto the ones of $\ell_1^k$, Problem \ref{ProblemJohnson} deals with maps taking values in the unit sphere of $\ell_2^k$ instead. This is however not an issue since  the Mazur maps introduced above give  uniform heomeomorphisms between  the unit spheres of $\ell_1$ and $\ell_2$ which are simultaneously support and step preserving. So, the second parts of both of these theorems remain true if one replaces $\ell_1^k$ by $\ell_2^k$ in their statements.
 
Compared with the proof of Theorem \ref{Thm.Intro.Supp.Card.Preserv}, the proof of Theorem \ref{Thm.Intro.Homeo.Equiv.Intro} contains only a minimal topological component, relying on the intermediate value theorem. The main new lemma, on which both results depend, is Lemma \ref{Lemma.Separation}. 


 Our methods also give rise to a new concentration inequality. The study of concentration inequalities in nonlinear geometry of Banach spaces started with  N.\ Kalton's introduction to Property $\mathcal Q$ (\cite[Section 4]{Kalton2007QJM}) and, right after,  with Kalton and N. Randrianarivony's theorem about the concentration of embeddings  of the Hamming graphs into asymptotically uniformly smooth Banach spaces (\cite[Theorem 4.2]{KaltonRandrianarivony2008MathAnnalen}). The investigation of these sort of phenomena has quickly become an important topic in Banach space geometry; we refer the reader to \cite{KaltonRandrianarivony2008MathAnnalen,LancienRaja2018Houston,BaudierLancienSchlumprecht2018,Braga2021Jussieu,BaudierLancienMotakisSchlumprecht2021JIMJ,Fovelle2025JFA}
 for some of the growing literature about concentration inequalities on Banach spaces and their relation to nonlinear Banach space geometry. However, in contrast to the papers just cited, our concentration inequality has  a different flavor since it is not between infinite dimensional Banach spaces, but rather about the asymptotic behavior of finite dimensional spaces whose dimensions are increasing. 
 
In order to state our concentration inequality precisely, we shall define the positive part of the spheres: \[S^+_{\ell_p^k}=\left\{(x_i)_{i=1}^k\in S_{\ell_p^k}\mid x_i\geq 0 \text{ for all }\ i\in \{1,\ldots, k\}\right\},\ \text{ for } \ p\in \{1,\infty\},\]  and, for $0=m_0<m_1<\ldots<m_d<k$ in $\N$, we write $\bar m=(m_1,\ldots,m_d)$ and 
\[z(\bar m)=\sum_{s=1}^d\left(1-\frac{s-1}{d}\right)1_{(m_{s-1},m_s]} \in S_{\ell_\infty^k}^+.\]

\begin{theorem}\label{ThmConcentration.LastSection.Intro}
Fix   $d\in \N$ and $\eps > 0$. Let $Q = \{a^{2j-1}\}_{j=1}^\infty \subseteq \N$, where  $a= \lceil 32/\eps + 2 \rceil$. Then, for all $m_1<\ldots< m_d<k\in Q$  and all step preserving maps $F:S_{\ell_\infty^k}^+ \to S_{\ell_1^k}^+$ with $\omega_F(\frac{1}{d})\le \frac{\eps}{8}$, we have
\[\Big\|F(z(\bar m))-\frac{1}{k}\sum_{i=1}^k e_i\Big\|_1\le \eps,\]
where $(e_i)_{i=1}^k$ denotes the standard basis of $\ell_1^k$. 
\end{theorem}



In Section \ref{SectionConcentration}, we discuss the relationship of this concentration property with Kalton's Property $\mathcal Q$. Morally, when our concentration inequality is put into the more general context of Banach spaces with unconditional bases, it can be seen as a sort of {\it local Property $\mathcal Q$} (see Definition \ref{Defi.Local.Prop.Q}), and its negation characterizes the standard basis of $c_0$ among all unconditional bases (see Corollary \ref{Corollary.Charactc0LocalPropQ}). \\

 At last, we conclude the introduction with a comment on the exposition. By the proof of Odell and Schlumprecht’s \cite[Theorem 2.1]{OdellSchlumprecht1994Acta}, Theorems \ref{Thm.Intro.Supp.Card.Preserv}, \ref{Thm.Intro.Homeo.Equiv.Intro}, and \ref{ThmConcentration.LastSection.Intro} may be reformulated with $\ell_1$ replaced by any Banach space $X$ that has an unconditional basis and non-trivial cotype, though the dimensional estimates must then be adjusted. By not relying on this theorem, we prove Theorems \ref{Thm.Intro.Homeo.Equiv.Intro} and \ref{ThmConcentration.LastSection.Intro} for the much broader class of unconditional bases that are not equivalent to the standard basis of $c_0$. Throughout, we give particular attention to the case where the target space is $\ell_r$ for $1\leq r<\infty$. We have nonetheless chosen to present our main results here specifically for $\ell_1$ since (1) the statements are cleaner in this setting, and (2) the motivation for this work arose from Problems \ref{ProblemPropH} and \ref{ProblemJohnson}.

\section{Examples of support and step preserving maps}\label{SectionExamples}

This short section contains some examples of families of maps satisfying the hypotheses of Theorems \ref{Thm.Intro.Supp.Card.Preserv}, \ref{Thm.Intro.Homeo.Equiv}, and \ref{ThmConcentration.LastSection} as well as an insight into some of the main ideas needed to prove these results.

Arguably, the most natural map between the spheres of  $\ell_\infty^k$ and  $\ell_1^k$ is given by
\begin{equation}\label{Eq.Maps.1.SecPrelim}F_k\colon (x_i)_{i=1}^k\in S_{\ell_\infty^k}\mapsto\left(\frac{  x_i}{|x_1|+\ldots+|x_k|}\right)_{i=1}^k\in S_{\ell_1^k}.
\end{equation}
The family $(F_k)_k$ is clearly both support and step preserving. However,  it is immediate to check that they are not equi-uniformly continuous. We can see this as follows:   let $(e_i)_i$ denote the canonical unit vectors in $\R^\N$ and, for each  $k\in\N$ and  $\delta>0$, let \[x(k,\delta)=e_1+\delta\sum_{i=2}^ke_i.\] So,  viewing each $x(k,\delta)$ as an element in $S_{\ell_\infty^k}$, we have that  $\|e_1-x(k,\delta)\|_\infty=\delta$ for all $k\in\N$ while 
\[\|F_k(e_1)-F_k(x(k,\delta))\|_1\geq \left|1-\frac{1}{1+(k-1)\delta}\right|\overset{k\to \infty}{\longrightarrow} 1.\]

The definition in \eqref{Eq.Maps.1.SecPrelim} can accommodate more complicated maps. Indeed, if $\varphi=(\varphi_i)_i$ is a sequence of maps from  $[-1,1]$ to $ \R$, then
\begin{equation}\label{Eq.Maps.2.SecPrelim}F^\varphi_k\colon (x_i)_{i=1}^k\in S_{\ell_\infty^k}\mapsto\left(\frac{  \varphi_i(x_i) }{|\varphi_1(x_1)|+\ldots+|\varphi_k(x_k)|}\right)_{i=1}^k\in S_{\ell_1^k}
\end{equation}
is well defined as long as at least one of the $\varphi_i(x_i)$ is not zero for all  $(x_i)_{i=1}^k$ in  $S_{\ell_\infty^k}$. 
If  $\varphi_i(0)=0$ for all $i\in\N$, then the maps $(F^\varphi_k)_k$ do not increase support. On the other hand, if all $\varphi_i$'s are the same, then $(F^\varphi_k)_k$ are step preserving. While the maps in \eqref{Eq.Maps.1.SecPrelim} are clearly homeomorphisms between the spheres, the same does not hold for \eqref{Eq.Maps.2.SecPrelim} and extra assumptions on $\varphi$ are needed. 

At last, perhaps one of the most interesting types of homeomorphisms between the spheres of $\ell_\infty^k$ and $\ell_1^k$ is given as follows. Given a subset  $A$ of $\R$, let $1_A$ denote its characteristic function. Then, for each $k\in\N$ and each $i\in \{1,\ldots, k\}$, let  
\[F_{k,i}\left((x_j)_{j=1}^k\right)=\int_{0}^1\frac{1_{[r,1]}(x_i)}{\sum_{j=1}^k1_{[r,1]}(x_j)}dr\]
for all $(x_j)_{j=1}^k\in S_{\ell_\infty^k}$. We then let 
\[F_k(x)=\left(F_{k,i}(x)\right)_{i=1}^k\ \text{ for all }\ x\in S_{\ell_\infty^k}.\]
So, $(F_k)_k$ are also both support and step preserving.

Evidently, each $F_{k+1}$ extends $F_k$. This can be used to define a topological homeomorphism $F:S_{c_0}^+\to S_{\ell_1}^+$, which admits the following explicit description: for non-increasing $x=(x_n)_{n=1}^\infty\in S_{c_0}^+$ and $y=(y_n)_{n=1}^\infty\in S_{\ell_1}^+$,
\[F(x) = \Big(\sum_{i=j}^\infty \frac{1}{i}\,(x_i-x_{i+1})\Big)_{j=1}^\infty \text{ and } F^{-1}(y) = \Big(jy_j + \sum_{i=j+1}^\infty y_i\Big)_{j=1}^\infty.\]
Since $F$ is by definition permutation equivariant, these formulae admit the obvious modifications for arbitrary $x=(x_n)_{n=1}^\infty\in S_{c_0}^+$ and $y\in S_{\ell_1}^+$. To the best of our knowledge, an explicit homeomorphism $F:S_{c_0}^+\to S_{\ell_1}^+$ (which extends in an elementary way to the spheres, and even to the Banach spaces $c_0$ and $\ell_1$) does not appear in the literature. The same formulae work verbatim for $c_0(\Gamma)$ and $\ell_1(\Gamma)$, for an arbitrary set $\Gamma$. A modification of this construction also yields an explicit topological homeomorphism $F:S_{L_1}^+\to S_{L_\infty}^+$, where the latter is equipped with convergence in measure.

The claim that these maps represent interesting homeomorphisms is (informally) further substantiated because, in some sense, they seem to be the closest we can get homeomorphisms between the spheres of $\ell_\infty^k$ and $\ell_1^k$ to be equi-uniformly continuous. Without getting into too many details (they will not be worthwhile since our main results show that these maps cannot be equi-uniformly continuous anyway), let us just say that these maps are ``equi-uniformly continuous around many points''. For instance, this holds around points of the form $1_A$, $A\subseteq \{1,\ldots, k\}$, and for interlacing elements, i.e., elements $(x_i)_{i=1}^k$ and $(y_i)_{i=1}^k$ in $S_{\ell_\infty^k}$ satisfying 
\[x_1\geq y_1\geq \ldots\geq x_k\geq y_k.\]
The lack of equi-uniform continuity for the maps $(F_k)_k$ can be noticed however when looking at elements of the form 
\[\sum_{s=1}^d\left(1-\frac{s-1}{d}\right)1_{(m_{s-1},m_s]} \ \text{ and } \ \sum_{s=1}^d\left(1-\frac{s-1}{d}\right)1_{(n_{s-1},n_s]},\]
where $(n_1,\ldots, n_d)$ and $(m_1,\ldots,m_d)$ are $d$-tuples of naturals such that $0=m_0=n_0<m_1<n_1<\ldots<m_d<n_d$ and which are ``very spread out''. We emphasize here that  this realization was actually the first step towards the main results in this article, and their proofs, as the reader will see below, still bear similarities with this approach.

\section{The key lemma}\label{SectionMainLemma}
The main result of this section is  Lemma \ref{Lemma.Separation}. This lemma will be essential for the proofs of all of our main theorems.

As mentioned in the introduction, many of our results apply to maps from the spheres of $\ell_\infty^k$ to the spheres of the canonical pavings of a given $1$-unconditional basis. We shall present our results here in this generality: throughout, $(X,\|\cdot\|_X)$ will be a Banach space with a normalized, $1$-unconditional basis $(e_i)_i$\footnote{Recall, a (Schauder) basis $(e_i)_i$ for a  Banach space $X$ is {\it normalized} if $\|e_i\|_X = 1$ for all $i \in \{1,\ldots, n\}$ and {\it $1$-unconditional} if $\|\sum_{i=1}^na_n e_i\|\leq \|\sum_{i=1}^nb_ne_i\|$ for all $n\in \N$ and all $a_1,\ldots,a_n,b_1,\ldots,b_n\in \R$ with $|a_i|\leq |b_i|$ for all $i\in \{1,\ldots, n\}$.} which is not equivalent to the standard basis of $c_0$ and, for each $k\in\N$, we let 
\[X_k=\mathrm{span}\{e_1,\ldots, e_k\}.\] The reader interested only in functions from $S_{\ell_\infty^k}$ to $S_{\ell_1^k}$ can assume throughout that $X=\ell_1$, $ (e_i)_i$ is the canonical basis of $\ell_1$, and $X_k=\ell_1^k$ for all $k\in\N$.

  Recall, a sequence $(x_i)_{i=1}^n$ in  $ X$ is called a {\it block sequence} (with respect to  $(e_i)_i$) if there are $s_1<\ldots< s_{n+1}\in \N$ such that 
\[x_i\in \mathrm{span} \{e_j\mid s_i\leq j <s_{i+1}\}\ \text{ for all }\ i\in \{1,\ldots, n\}.\]
We fix two values $q\leq p\in[1,\infty]$, with $q<\infty$, for which $X$ satisfies {\it lower $p$ and upper $q$ estimates on block sequences with constant one}, that is, for every block sequence $(x_i)_{i=1}^n$ in $X$,
\begin{equation}\label{EqLowerUpperqpEstimates}\Big(\sum_{i=1}^n\|x_i\|_X^p\Big)^{1/p} \leq \Big\|\sum_{i=1}^nx_i\Big\|_X\leq \Big(\sum_{i=1}^n\|x_i\|_X^q\Big)^{1/q},
\end{equation}
with the obvious modifications when $p=\infty$. While this assumption holds automatically for the choice $q=1$ and $p=\infty$, the proximity of $p$ and $q$ plays a role in the dimensional estimates in our theorems. For instance, when $X = \ell_r$ and $(e_i)_i$ is its standard basis, we have $p=q=r$, which yields the strongest bounds. Throughout, we consider the function $\psi \colon \N\to [0,\infty)$ given by \[\psi(k)=\Big\|\sum_{i=1}^{k}e_i\Big\|_X\ \text{ for all }\ k\in\N.\] It follows from the $1$-unconditionality assumption of  $(e_i)_i$ that  $\psi$ is a non decreasing function and it follows from the assumption that $(e_i)_i$ is not equivalent to the standard $c_0$ basis that  $\lim_{k \to \infty}\psi(k)=\infty$. We linearly identify $\bigcup_k X_k$ with $c_{00}(\N)$ by mapping the basis $(e_i)_{i=1}^\infty$ to the standard basis of $c_{00}(\N)$ --- here $c_{00}(\N)$ denotes the space of finitely supported functions $\N\to \R$. Under under this identification, the formula for $\psi(k)$ becomes
\begin{equation*}
    \psi(k) = \|1_{[1,k]}\|_X.
\end{equation*}

We emphasize now a notation that will be used throughout the remainder of this article. If $F\colon S_{\ell_\infty^k}\to S_{X_k}$ is any map, we write $F=(F_i)_{i=1}^k$, meaning that each $F_i$ is the composition of $F$ with the real-valued map $x \in S_{X_k} \mapsto x_i\to \R$ giving the coefficient of $e_i$ in the basis expansion of $x$. Also, for each $d\in \N$ and each subset $M \subseteq \N$, we let 
\[[M]^d=\left\{(n_1,\ldots,n_d)\in M^d\mid n_1<\ldots<n_d\right\}.\]
So, $[M]^d$ is the set of subsets of $M$ with $d$ elements, and we write its elements as $d$-tuples in increasing order. 

Given a subset $P \subseteq \N$, we define the function $\psi_P: \N \to [0,\infty)$ by
\begin{equation*}
    \psi_P(k) := \min\left\{\|1_{[1,k]\cap P}\|_X,\|1_{[1,k]\cap P^c}\|_X\right\},
\end{equation*}
where $P^c $ denotes the complement of $P$ in $\N$. 
Note that, by $1$-unconditionality, $\psi_P(k)\leq \psi(k)$ for all $k\in\N$. In fact, $P$ can be chosen so that these terms are equivalent in the following sense:

\begin{lemma}\label{LemmaPartition}
There exists $P \subseteq\N$ such that, for all $k\in\mathbb{N}$,
\[\psi_P(k)\geq \frac{1}{2}(\psi(k)-1).\]
\end{lemma}

\begin{proof}
We algorithmically decide, for $k\in\N$, whether $k$ will be admitted into $P$ or into $P^c$. Initially, we admit $1$ into $P$ and $2$ into $P^c$. Assuming this has been decided for $1,\dots,k-1$, we admit $k$ into $P$ if
\[\|1_{([1,k-1]\cap P)\cup\{k\}}\|_X \leq  \|1_{([1,k-1]\cap P^c)\cup\{k\}}\|_X,\]
and otherwise we admit $k$ into $P^c$. This completes the definition of $P$ (and consequently of $P^c$).

Notice that
\begin{equation}\label{Eq.PPCatmost1}\Big|\|1_{[1,k]\cap P}\|_X-\|1_{[1,k]\cap P^c}\|\Big|\leq 1\end{equation}
for all $k\in\N$. This follows from induction on $k$: if $k=1$ or $k=2$, the inequality is trivial, suppose then that it holds for $k-1$ and let us show it also holds for $k$, assuming $k \geq 3$. If $k\in P$, then $k\not\in P^c$ and it follows from the definition of $P$ and the fact that $(e_i)_i$ is normalized that  
\[\|1_{[1,k]\cap P}\|_X-\|1_{[1,k]\cap P^c}\|_X\leq \|1_{([1,k-1]\cap P^c)\cup\{k\}}\|_X-\|1_{[1,k-1]\cap P^c}\|_X\leq \|e_{k}\|_X= 1.\]
On the other hand, as $(e_i)_i$ is $1$-unconditional, the inductive hypothesis gives 
\[\|1_{[1,k]\cap P^c}\|_X-\|1_{[1,k]\cap P}\|_X\leq \|1_{[1,k-1]\cap P^c}\|_X-\|1_{[1,k-1]\cap P}\|_X\leq 1.\]
In case $k\in P^c$ the argument is analogous.

As    $1_{[1,k]\cap P}+1_{[1,k]\cap P^c} = 1_{[1,k]}$,   \eqref{Eq.PPCatmost1} gives that \[\psi(k)\leq \|1_{[1,k]\cap P}\|_X+\|1_{[1,k]\cap P^c}\|_X\leq 2\|1_{[1,k]\cap P}\|_X+1. \]
So, 
\[\|1_{[1,k]\cap P}\|_X \geq \frac{1}{2}(\psi(k) - 1).\]
The same estimate holds with $P$ replaced by $P^c$.
\end{proof}

\begin{lemma}\label{Lemma.Bounds.Max.Alpha.Beta(bis)}
Let $\eps>0$ and $d\in\N$. Let $P \subseteq \N$ be a subset satisfying the conclusion of Lemma~\ref{LemmaPartition}, and let $M = \{k_j\}_{j=1}^\infty$ be a subset of $\N$ satisfying the growth condition
\begin{equation} \label{eq:Mgrowth}
 \psi(k_{j+1}) \geq \Big(\frac{8d^{1/q-1/p}}{\eps}+2\Big)\psi(k_j)+1.
\end{equation}
Then for all ${\bar m},{\bar n}\in[M]^d$ with $m_1<n_1<\ldots<m_d<n_d$ and all $(\lambda_s)_{s=1}^d$ in $\mathbb{R}$, the following implications hold: For each $Q\in \{P,P^c\}$,
\begin{align*}
 \Big\|\sum_{s=1}^d\lambda_s 1_{(n_{s-1},n_s]\cap Q}\Big\|_X \leq 1 &\implies \Big\|\sum_{s=1}^d \lambda_s1_{(n_{s-1},m_s]}\Big\|_X \leq \frac{\eps}{4}, \\
 \Big\|\sum_{s=1}^d\lambda_s 1_{(m_{s-1},m_s]\cap Q}\Big\|_X \leq 1 &\implies \Big\|\sum_{s=1}^d \lambda_s1_{(m_{s-1},n_{s-1}]}\Big\|_X \leq \frac{\eps}{4}.
\end{align*}
\end{lemma}

\begin{proof}
First notice that, by Lemma \ref{LemmaPartition} and the growth condition on $(\psi(k_j))_{j=1}^\infty$, we have that, for   all $\ell<m<n\in M$,
\[\frac{\psi(m)}{\psi_P(n) - \psi(\ell)}\leq \frac{\psi(m)}{\psi_P(n) - \psi(m)} \leq \frac{2\psi(m)}{\psi(n)-1-2\psi(m)} \leq \frac{\varepsilon}{4d^{1/q-1/p}}.\] 
We only demonstrate the first implication; the second is established similarly by replacing $n_{s-1}<m_s<n_s$ with $m_{s-1}<n_{s-1}<m_s$.

Using that $(e_i)_i$ is $1$-unconditional and the lower $p$ and upper $q$ estimates on block sequences satisfied by $(e_i)_i $ (see   \eqref{EqLowerUpperqpEstimates}), we have
\begin{align*}
\Big\|\sum_{s=1}^d\lambda_s1_{(n_{s-1},m_s]}\Big\|_X &\leq \Big(\sum_{s=1}^d \|\lambda_s1_{(n_{s-1},n_s]\cap Q}\|_X^q\Big(\frac{\|1_{(n_{s-1},m_s]}\|_X}{\|1_{(n_{s-1},n_s]\cap Q}\|_X}\Big)^q\Big)^{1/q}\\
&\leq  \Big(\sum_{s=1}^d \|\lambda_s1_{(n_{s-1},n_s]\cap Q}\|^q_X\Big(\frac{\psi(m_s)}{\psi_P(n_s)-\psi(n_{s-1})}\Big)^q\Big)^{1/q}\\
&\leq\frac{\varepsilon}{4d^{1/q-1/p}} \Big(\sum_{s=1}^d \|\lambda_s1_{(n_{s-1},n_s]\cap Q}\|_X^q\Big)^{1/q}\\
&\leq \frac{\varepsilon}{4}\Big(\sum_{s=1}^d\|\lambda_s1_{(n_{s-1},n_s]\cap Q}\|_X^p\Big)^{1/p}\\
& \leq \frac{\varepsilon}{4}\Big\|\sum_{s=1}^d\lambda_s 1_{(n_{s-1},n_s]\cap Q}\Big\|_X\\
&\leq \frac{\eps}{4}
\end{align*}
(here we use that $q\leq p$ and, consequently, $d^{1/q-1/p}\geq 1$).
\end{proof}

For $d\in\N$ and $\eps>0$, let us say that subsets $P \subseteq \N$ and $M = \{k_j\}_{j=1}^\infty \subseteq \N$ satisfy {\bf Assumption} $\boldsymbol{A(d,\eps)}$ if
\begin{align} \label{AssumptionA}
    \bullet &\text{ for all }k\in\mathbb{N},\; \psi_P(k)\geq 1/2(\psi(k) - 1), \\
    \bullet &(k_j,k_{j+1}]\cap P \neq \emptyset \text{ and } (k_j,k_{j+1}]\cap P^c\neq \emptyset \text{ for all } j \geq 1, \text{ and} \notag \\
    \bullet &M \text{ satisfies the growth hypothesis \eqref{eq:Mgrowth} of Lemma~\ref{Lemma.Bounds.Max.Alpha.Beta(bis)}:} \notag\\
    &\hspace{.5in}  \psi(k_{j+1}) \geq \Big(\frac{8d^{1/q-1/p}}{\eps}+2\Big)\psi(k_j)+1. \notag
\end{align}
Such a set $P$ always exists by Lemma~\ref{LemmaPartition}, and such a set $M$ always exists since $\lim_{k\to\infty} \psi(k) = \infty$ and since $P,P^c$ are (necessarily) infinite.

\begin{remark} \label{rmk:AssumptionAl1}
Note that if $X = \ell_r$, $1\leq r<\infty$,  with its standard basis, then $\psi_{2\N}(k) \geq \frac{1}{2}(\psi(k)-1)$. Hence, the choices $P = 2\N$ and
\begin{equation*}
    M = \left\{a^{j-1}\right\}_{j=1}^\infty, \hspace{.2in}\text{where}\ a := \lceil (8/\eps + 3)^r \rceil
\end{equation*}
satisfy {\bf Assumption} $\boldsymbol{A(d,\eps)}$. 
\end{remark}

Let us set some notation to describe elements of $S_{\ell_\infty^k}$ that play an essential role in our arguments. For $d\in \N$, $\eps >0$, $u=((a_s)_{s=1}^d,(b_s)_{s=1}^d,c) \in S_{\ell_\infty^{2d+1}}$, subsets $P \subseteq \N$ and $M \subseteq \N$ satisfying {\bf Assumption} $\boldsymbol{A(d,\eps)}$, if $\bar m \in [M]^d$ and $k>m_d$, we set
\begin{equation} \label{Definition.x.m.u.k}
    x(\bar m,u,k)=\sum_{s=1}^d a_s1_{(m_{s-1},m_s]\cap P}+\sum_{s=1}^d b_s1_{(m_{s-1},m_s]\cap P^c}+c1_{(m_d,k]} \in S_{\ell_\infty^k}.
\end{equation}
Note that $\|x(\bar m,u,k)\|_\infty$ indeed equals 1 since $\|u\|_{\infty} = 1$, $(m_{s-1},m_s] \cap P \neq \emptyset$,  $(m_{s-1},m_s] \cap P^c \neq \emptyset$ for all $s\in \{1,\dots d\}$, and $(m_d,k]\neq \emptyset$. Furthermore, denoting $a_{d+1}=b_{d+1} = c$, it is easy to see that for $\bar m,\bar n \in [\N]^d$ with $m_1<n_1<\cdots<m_d<n_d$, we have the distance bound
\begin{equation} \label{eq:d(interlacing)}
    \|x(\bar m,u,k) - x(\bar n,u,k)\|_\infty \leq \max_{1\leq s\leq d}\max\{|a_s - a_{s+1}|,|b_s-b_{s+1}|\}.
\end{equation}

The following is our key lemma.

\begin{lemma}\label{Lemma.Separation}
Let $d\in\N$ and $\eps > 0$, and let $P \subseteq \N$ and $M \subseteq \N$ be subsets satisfying {\bf Assumption} $\boldsymbol{A(d,\eps)}$. Let $\bar m,\bar n\in [M]^d$ such that $m_1<n_1<\ldots<m_d<n_d$, let $k> n_d$, let $u=((a_s)_{s=1}^d,(b_s)_{s=1}^d,c) \in S_{\ell_\infty^{2d+1}}$, and let $F = (F_i)_{i=1}^k: S_{\ell_\infty^k} \to S_{X_k}$ be a step preserving function with $F_{i}(x(\bar m,u,k))=0$ for all $i\in (m_d,k]$, where $x(\bar m,u,k) \in S_{\ell_\infty^k}$ is defined in \eqref{Definition.x.m.u.k}. Then 
\[\|F(x(\bar m,u,k))-F(x(\bar n,u,k))\|_X>1-\eps.\] 
\end{lemma}

\begin{proof}
Since $F$ is step preserving, it follows from the definition of $x(\bar m,u,k)$ that there are
\[(\alpha_s(\bar m))_{s=1}^d, (\beta_s(\bar m))_{s=1}^d,(\alpha_s(\bar n))_{s=1}^d, (\beta_s(\bar n))_{s=1}^d, \gamma(\bar n) \in [-1,1]^d\]
such that, for all $s\in\{1,\dots d\}$,
\begin{enumerate}
    \item $\alpha_s(\bar m)=F_{i}(x(\bar m,u,k))$ for all $i\in (m_{s-1},m_s] \cap P$, 
    \item $\beta_s(\bar m)=F_{i}(x(\bar m,u,k))$ for all $i\in (m_{s-1},m_s] \cap P^c$,
    \item $\alpha_s(\bar n)=F_{i}(x(\bar n,u,k))$ for all $i\in (n_{s-1},n_s] \cap P$, 
    \item $\beta_s(\bar n)=F_{i}(x(\bar n,u,k))$ for all $i\in (n_{s-1},n_s] \cap P^c$, and
    \item $\gamma(\bar n) = F_{i}(x(\bar n,u,k))$ for all $i\in (n_d,k]$.
\end{enumerate}
Using the above equations and 1-unconditionality of $(e_i)_i$, we have that
\begin{align*}\label{Eq.Sum.At.Most.1}
    1 &= \|F(x(\bar m,u,k))\|_X\\
      &\geq\max\Big\{\Big\|\sum_{s=1}^d\alpha_s(\bar m)1_{(m_{s-1},m_s]\cap P}\Big\|_X,\Big\|\sum_{s=1}^d\beta_s(\bar m)1_{(m_{s-1},m_s]\cap P^c}\Big\|_X\Big\}. \notag
\end{align*}
The same is true for $\bar n$:
\begin{align*}
    1 \geq\max\Big\{\Big\|\sum_{s=1}^d\alpha_s(\bar n)1_{(n_{s-1},n_s]\cap P}\Big\|_X,\Big\|\sum_{s=1}^d\beta_s(\bar n)1_{(n_{s-1},n_s]\cap P^c}\Big\|_X\Big\}. \notag
\end{align*}
Therefore, by Lemma \ref{Lemma.Bounds.Max.Alpha.Beta(bis)} and $1$-unconditionality, the quantities  
\begin{enumerate} [label=(\roman*)]
   
    \item \label{Item.i} $\Big\|\sum_{s=1}^d\alpha_s(\bar m)1_{(m_{s-1},n_{s-1}]\cap P}\Big\|_X$
   
    \item \label{Item.ii} $\Big\|\sum_{s=1}^d\beta_s(\bar m)1_{(m_{s-1},n_{s-1}]\cap P^c}\Big\|_X$
   
    \item \label{Item.iii} $\Big\|\sum_{s=1}^d\alpha_s(\bar n)1_{(n_{s-1},m_s]\cap P}\Big\|_X$, and
   
    \item \label{Item.iv} $\Big\|\sum_{s=1}^d\beta_s(\bar n)1_{(n_{s-1},m_s]\cap P^c}\Big\|_X$

\end{enumerate}
are all at most $\eps/4$. Since $\|F(x(\bar m,u,k))\|_X=1$ and $F_{i}(x(\bar m,u,k))=0$ for all $i\in (m_d,k]$, it follows from   \ref{Item.i} and \ref{Item.ii} above that
\[\Big\|\sum_{s=1}^d \alpha_s(\bar m)1_{(n_{s-1},m_s] \cap P}+\beta_s(\bar m)1_{(n_{s-1},m_s] \cap P^c}\Big\|_X \ge 1-\frac{\eps}{2}.\]
Therefore, using the $1$-unconditionality of $(e_i)_i$, together with \ref{Item.iii} and \ref{Item.iv}, we conclude that
\begin{align*}
\|&F(x(\bar m,u,k))- F(x(\bar n,u,k))\|_X \\
&\ge \Big\|\sum_{s=1}^d (\alpha_s(\bar m)-\alpha_s(\bar n))1_{(n_{s-1},m_s] \cap P}+(\beta_s(\bar m)-\beta_s(\bar n))1_{(n_{s-1},m_s] \cap P^c}\Big\|_X\\
&\ge 1-\eps.
\end{align*}
\end{proof}

Specializing Lemma~\ref{Lemma.Separation} to the case of  $X=\ell_r$, $1\leq r<\infty$ with its standard basis yields the following proposition, which we employ in Section~\ref{SectionMapsNotIncreasingSuppSize}.

\begin{prop}\label{Prop.l1Separation}
Let $d\in\N$ and $\eps>0$. Let $M = \{k_j\}_{j=1}^\infty \subseteq \N$, where $k_j := a^{j-1}$ and $a= \lceil (8/\eps + 3)^r \rceil$. Let $\bar m,\bar n\in [M]^d$ such that $m_1<n_1<\ldots<m_d<n_d$, let $k > n_d$, let $u=((a_s)_{s=1}^d,(b_s)_{s=1}^d,c) \in S_{\ell_\infty^{2d+1}}$, and let $F = (F_i)_{i=1}^k: S_{\ell_\infty^k} \to S_{\ell_r^k}$ be a step preserving function with $F_{i}(x(\bar m,u,k))=0$ for all $i\in (m_d,k]$, where $x(\bar m,u,k)$ is defined as in \eqref{Definition.x.m.u.k}. Then 
\[\|F(x(\bar m,u,k))-F(x(\bar n,u,k))\|_r>1-\eps.\]
\end{prop}

\begin{proof}
This follows immediately from Lemma~\ref{Lemma.Separation} and the fact that $P=2\N$ and $M = \{k_j\}_{j=1}^\infty$ satisfy {\bf Assumption} $\boldsymbol{A(d,\eps)}$ by Remark~\ref{rmk:AssumptionAl1}.
\end{proof}

\section{Maps which do not increase support  size}\label{SectionMapsNotIncreasingSuppSize}

We shall now isolate a simple consequence of Proposition~\ref{Prop.l1Separation} which will be enough for us to rule out the existence of equi-uniformly continuous homeomorphisms from $S_{\ell_\infty^k}$ to $S_{\ell_1^k}$ which preserve support (Theorem \ref{Thm.Main}) or, more generally, which do not increase support size (Theorem \ref{Thm.Intro.Supp.Card.Preserv}). 

We now introduce a notation which, in this simpler scenario, will have the same role as $x(\bar m,u,k)$ from \eqref{Definition.x.m.u.k}.
Given $d\in \N$  and $\bar m\in [\N]^d$,  
we define 
\begin{equation}\label{Eq.Def.Z}
z(\bar m)=\sum_{s=1}^d \left(1-\frac{s-1}{d}\right) 1_{(m_{s-1},m_s]}.
\end{equation}
If $k\geq m_d$, we consider $z(\bar m)$ as an element of $S_{\ell^k_\infty}$ in the canonical way. In this case, the graph of $z(\bar m)$ is simply a staircase descending from $1$ to $0$.

We say that a function $F\colon S_{\ell_\infty^k}\to S_{\ell_1^k}$ between subsets of vector spaces   is {\it{support  preserving}} if 
\begin{equation}\label{Eq.Supp.Containment}\mathrm{supp}(F(x))= \mathrm{supp}(x)\ \text{ for all } x\in S_{\ell_\infty^k}.
\end{equation}

The next result is a simple corollary of Proposition~\ref{Prop.l1Separation},   stated for $X=\ell_r$, $1\leq r<\infty$.

\begin{lemma}
Let $d\in\N$ and $\eps>0$. Let $M = \{k_j\}_{j=1}^\infty \subseteq \N$, where $k_j := a^{j-1}$ and $a= \lceil (8/\eps + 3)^r \rceil$. Let $\bar m,\bar n\in [M]^d$ such that $m_1<n_1<\ldots<m_d<n_d$ and let $k > n_d$. Let $F: S_{\ell_\infty^k}\to S_{\ell_r^k}$ be a support and step preserving map. Then
\[\|F(z(\bar m))-F(z(\bar n))\|_r>1-\eps.\]\label{Lemma.4.1.Simple}
\end{lemma}

\begin{proof}
Fix $d\in\N$ and $\eps>0$. In order to apply Proposition~\ref{Prop.l1Separation}, we let 
\[a_s=b_s=1-\frac{s-1}{d} \text{  and }c=0,\]
for all $s\in \{1,\ldots, d\}$. In this case, for $u=((a_s)_{s=1}^d, (b_s)_{s=1}^d,c)$, $x(\bar m, u, k)$ equals $z(\bar m)$.
As $F$ is support preserving and the coordinates of $z(\bar m)$ in $(m_d,k]$ are zero, we have that $F_{i}(z(\bar m))=0$ for all $i\in (m_d,k]$. In other words, the hypotheses of Proposition~\ref{Prop.l1Separation} on $\bar m$ are always satisfied automatically. The result then follows from that proposition.
\end{proof}

In order to be able to use Lemma \ref{Lemma.4.1.Simple} for support preserving maps, we shall now show how to modify a support preserving map  $S_{\ell_\infty^k}\to S_{\ell_1^k}$ in a way that makes it step preserving as well  without changing its modulus of continuity. In fact, we will actually show that it is possible to replace support preserving maps $S_{\ell_\infty^k}\to S_{\ell_1^k}$ by maps satisfying a stronger equivariance property. For this, we need some definitions. 
For each $k\in\N$, let $\mathrm{Per}_k$ denote the group of permutations of $\{1,\ldots, k\}$. For each $\pi\in \mathrm{Per}_k$, let $P_\pi\colon \R^k\to \R^k$ be the map given by 
  \[P_\pi\left((x_j)_{j=1}^k\right)=\left(x_{\pi(j)}\right)_{j=1}^k\]
  for all $(x_j)_{j=1}^k\in \R^k$.  We say that a map $F\colon S_{\ell^k_\infty}\to S_{\ell_1^k}$ is {\it{equivariant with respect to basis permutations}} if 
  \[F(P_\pi(x))=P_\pi(F(x))\text{ for all }x\in S_{\ell_\infty^k}\text{  and all } \pi \in \mathrm{Per}_k.\]

\begin{lemma}\label{Lemma.Equivariant.Injective}
Let $k\in\N$ and $F=(F_i)_{i=1}^k\colon S_{\ell_\infty^k}\to S_{\ell_1^k}$ be  equivariant with respect to basis permutation. Then $F$ is step preserving.
\end{lemma}

\begin{proof} Let $\pi$ be the permutation of $\{1,\ldots, k\}$ which takes $i$ to $j$ and all other elements to themselves. If $x_i=x_j$, then
\[F(x)=F(P_\pi(x))=P_\pi(F(x)),\] and we are done.
\end{proof}

For the next lemma, recall that 
$S^+_{\ell_1^k}$ denotes the positive part of $S_{\ell_1^k}$. Note that  $S^+_{\ell_1^k} $ is a convex facet of  $S_{\ell_1^k}$.

\begin{lemma}\label{LemmaSupportPreservingToPermutationPreserving}
 Let $k\in\N$ and $F\colon S_{\ell_\infty^k}\to S_{\ell_1^k}^+$ be any map. Then, the map $G\colon S_{\ell_\infty^k} \to S_{\ell_1^k}^+$ given by
    \[G=\frac{1}{k!}\sum_{\pi\in \mathrm{Per}_k} P_{\pi^{-1}}\circ F\circ P_\pi\]
is equivariant with respect to basis permutation and the modulus of uniform continuity of $G$ is bounded above by that of $F$, i.e., $\omega_G\leq \omega_F$. Furthermore, if the map $F$ is support   preserving, then so is $G$.
\end{lemma}

\begin{proof}
First notice that, as $S_{\ell_1^k}^+$ is convex  and permutation invariant, the image of  $G$ is indeed inside $S_{\ell_1^k}^+$, i.e., $G$ is well defined. The fact that $G$ is equivariant with respect to basis permutation is straightforward from its construction and  so is the fact that $G$ is support   preserving if one assumes the same for $F$. Moreover, as the maps $P_\pi$, for  $\pi\in \mathrm{Per}_k$, are isometries when considered as operators on both $\ell_\infty^k$ and $\ell_1^k$, it is immediate that $\omega_G\leq \omega_F$. 
\end{proof}

\begin{theorem}\label{Thm.Main}
Let $d \in \N$. Then for any $ k \geq 19^{2d-1}+1$ and any support preserving map   $F:S_{\ell_\infty^k}\to S_{\ell_1^k}$, we have $\omega_F(\frac1d)\ge \frac12$. In particular, there is no sequence $(F_k:S_{\ell_\infty^k}\to S_{\ell_1^k})_{k=1}^\infty$ of equi-uniformly continuous support preserving maps.
\end{theorem}

\begin{proof}
Let  $k \geq 19^{2d-1}+1$ and let $F:S_{\ell_\infty^k}\to S_{\ell_1^k}$ be a support preserving map. As the map \[(x_1,\ldots,x_k)\mapsto (|x_1|,\ldots,|x_k|)\] is $1$-Lipschitz, the function \[\tilde F(x) = (|F_{1}(x)|,\ldots,|F_{k}(x)|),\] has modulus of uniform continuity bounded above by that of $F$, is support preserving like $F$, and has the additional property that  $\tilde F[S_{\ell_\infty^k}]\subset S_{\ell_1^k}^+$. Hence, it suffices to prove the theorem assuming that $F[S_{\ell_\infty^k}]\subset S_{\ell_1^k}^+$.

By Lemma \ref{LemmaSupportPreservingToPermutationPreserving}, we can assume further more that $F$ is   equivariant with respect to basis permutation. By Lemma \ref{Lemma.Equivariant.Injective}, $F$ is step preserving. By Lemma \ref{Lemma.4.1.Simple}, applied with $\eps=\frac12$, we see that there are $\bar m,\bar n\in [M]^d$ with $m_1<n_1<\ldots<m_d<n_d$ such that  $n_d < 19^{2d-1}+1 \leq k$ and
\[\|F(z(\bar m))-F(z(\bar n))\|_1\geq \frac12.\] 
Finally, we have by \eqref{eq:d(interlacing)} that $\|z(\bar m)-z(\bar n)\|_\infty=\frac1d$. This concludes the proof. 
\end{proof}

Notice that Theorem \ref{Thm.Main} requires only the preservation of supports. In particular, it applies to maps that are neither injective nor continuous. In contrast, when considering maps between the spheres of finite-dimensional Banach spaces of the same dimension, injectivity and continuity alone guarantee surjectivity. This fact will be used repeatedly throughout the paper. It follows from Brouwer’s invariance of domain theorem together with the fact that the $k$-dimensional Euclidean sphere is not null-homotopic. We record it here for later reference.

\begin{prop}
\label{invariance of domain}
Let $k\in\mathbb N$. Any continuous and injective map between the spheres of two $k$-dimensional Banach spaces is surjective.\qed
\end{prop}


 \begin{lemma}\label{LemmaCardOfSupportPreservingToSupportPreserving}
Let  $k\in\N$ and $F\colon S_{\ell^k_\infty}\to S_{\ell_1^k} $ be an injective continuous map which does not increase support size. Then, there is a $\pi\in \mathrm{Per}_k$ such that $F\circ P_\pi$ is  support preserving.
In particular, such $F$   preserves support size. 
\end{lemma}

\begin{proof}
Let $(e_i)_{i=1}^k$ be the canonical basis of $\R^k$. As $F$ does not increase support size, $F$ takes elements whose support have cardinality $1$ to elements with the same property. As there are only $2k$ such elements  in both $S_{\ell_\infty^k}$ and $S_{\ell_1^k}$, the injectivity of $F$ implies that \begin{equation}\label{Eq.FTakesSupp1ToSupp1}
F(\{\lambda e_i\mid i\leq k,\ \lambda \in \{-1,1\}\})=\{\lambda e_i\mid i\leq k,\ \lambda \in \{-1,1\}\}.
\end{equation}   Using the injectivity of $F$ again,  this together with the fact that $F$ does not increase support size, implies that  
\begin{equation}\label{Eq.Supp2ToSupp2}\forall x\in S_{\ell_\infty^k},\ |\mathrm{supp}(x)|=2\ \Rightarrow\ |\mathrm{supp}(F(x))|=2.\end{equation}

\begin{claim}
For all $i\in \{1,\ldots, k\}$, $F(-e_i)=-F(e_i)$.
\end{claim}

\begin{proof} Assuming that our claim is wrong, we could find, using \eqref{Eq.FTakesSupp1ToSupp1}, $i,j,l\in\{1,2,\ldots k\}$, $i\not=j$, and $\lambda,\lambda’\in\{-1,1\}$ so that   $F(e_i)=\lambda e_l$ and $F(\lambda’ e_j)=-\lambda e_l$.
As $i\neq j$, there is a continuous path $\varphi: [0,1] \to  \text{span}\{e_i,e_j\}\cap S_{\ell_\infty^n}$, from $e_i$ to $\lambda’e_j$ such that, for all $t \in (0,1)$, the support of $\varphi(t)$ has exactly two elements: simply let $\varphi$ be the union of the line segment from $e_i$ to $e_i+\lambda' e_j$ and the line segment from $e_i+\lambda'e_j$ to $\lambda'e_j$. By \eqref{Eq.Supp2ToSupp2},   $F\circ \varphi$ is a continuous path such that $F( \varphi(0))=\lambda e_\ell$, $F(\varphi(1))=-\lambda e_\ell$, and $|\mathrm{supp}(F(\varphi(t)))|=2$ for all $t\in (0,1)$. Since constructing such a continuous path in $S_{\ell_1^k}$ is not possible, we have reached a contradiction.
\end{proof}

The previous claim together with \eqref{Eq.FTakesSupp1ToSupp1} imply that  there is a permutation $\pi\in \mathrm{Per}_k$ and a choice of signs $(\lambda_i)_{i=1}^k\in \{-1,1\}^k$ such that 
\[F(e_{\pi(i)})=\lambda_i e_{i}\text{ and }F(-e_{\pi(i)})=-\lambda_i e_{i}\text{ for all } i\leq k.\]
 Letting  $G=F\circ P_\pi$, we have 
\begin{equation}\label{EqGSupp1Good}G(  e_i)=\lambda_i e_{i}\text{ and }G(-e_i)=-\lambda_i e_{i}\ \text{ for all }i\leq k.\end{equation}
In other words, $G$   preserves the support of elements with support of cardinality $1$. It is also clear that $G$ also does not increase support size. Write $G=(G_{i})_{i=1}^k$ for the coordinate functions of  $G$.

Let us show  that $G$ is support preserving. For that, let
 \begin{align*} 
s=\max  \{i\in \{1,\ldots, k\}\mid \ &  |\mathrm{supp}(x)|\leq i \ \text{ implies }\ \mathrm{supp}(G(x))= \mathrm{supp}(x)  \} \notag.\end{align*} 
By \eqref{EqGSupp1Good},  $s$ is well defined, i.e., $s\geq 1$. We will be done if we show that  $s=k$. Suppose towards a contradiction that $s<k$ and pick   $x\in S_{\ell_\infty^k}$ with  $\mathrm{supp}(x)\neq \mathrm{supp}(G(x)) $ so that   $|\mathrm{supp}(x)|=s+1$. As $G$ does not increase support size, we also have   $|\mathrm{supp}(G(x))|\leq s+1$. 

Let us notice that, by  Proposition \ref{invariance of domain}, $ |\mathrm{supp}(G(x))|= s+1$. Indeed, given $I\subseteq \{1,\ldots, k\}$ and $p\in \{1,\infty\}$, write 
\[S_{\ell_p^k}(I)=\left\{(x_i)_{i=1}^k\in S_{\ell_p^k}\mid x_i=0 \text{ for all } i\not\in I\right\}.\]
So, $S_{\ell_p^k}(I)$ is homeomorphic to the sphere of a $|I|$-dimensional space. By the definition of $s$, for all $I\subseteq \{1,\ldots, k\}$ with $|I|\leq s$, we have that \[F\left[S_{\ell_\infty^k}(I)\right]\subseteq S_{\ell_1^k}(I). \]
Since $F$ is injective and continuous, this inclusion must actually be an equality. Hence, if  $|\mathrm{supp}(G(x))|<s+1$, then $G(x)$ must be the image of an element in $S_{\ell_\infty^k}(I)$ with $I=\mathrm{supp}(G(x))$. This is however not the case since $\mathrm{supp}(x)=s+1$ and $F$ is injective. We pointout that there is nothing special about $x$ here. Our argument shows that $|\mathrm{supp}(G(z))|=s+1$ for all $z$ with $|\mathrm{supp}(z)|=s+1$. This will be used later in the proof.

If $s+1=k$, then we have no choice but that both $\mathrm{supp}(x)$ and  $  \mathrm{supp}(G(x)) $ equal $\{1,\ldots, k\}$.   Since, by our choice of $x$, these supports do not coincide, $s+1<k$. As $\mathrm{supp}(G(x))\neq\mathrm{supp}(x)$,  $|\mathrm{supp}(x)|=|\mathrm{supp}(G(x))|$ and $|\mathrm{supp}(x)|<k$, there must be $i\in \mathrm{supp}(x)$ and $j\notin \mathrm{supp}(x)$ such that 
\begin{equation}\label{EqGzeroatiforx}G_{i}(x)=0\text{ and } G_{j}(x)\neq 0.
\end{equation}
As the support of $x$ contains at least two elements, pick $\ell\in \mathrm{supp}(x)$ different than $i$ and let $y\in S_{\ell_\infty^k}$ be such that 
\[\mathrm{supp}(y)=\mathrm{supp}(x)\setminus \{\ell\}.\]
For didactic purposes, let us fix this choice of $y$:  unless $x$'s $\ell$th  coordinate is the only with absolute value  $1$, we  take   $y$ to be the same as $x$ except in the $\ell$th coordinate, in which $y$ must be zero. In   case $x$'s $\ell$th coordinate is the only whose absolute value is $1$, $y$ is chosen to be $x$  except in the  coordinates $\ell$ and $i$: in its $\ell$th coordinate we make   $y$ zero and in its $i$th coordinate we make it be  the sign of $x$'s $i$th coordinate. 
As $|\mathrm{supp}(y)|=s$, our choice of $s$ implies that 
\begin{equation}\label{Eq.suppy=suppGy}
    \mathrm{supp}(G(y))=\mathrm{supp}(y).
    \end{equation}

 With the choice of $y$ above,  we can pick a continuous   $\varphi:[0,1]\to S_{\ell_\infty^k}$ such that
\begin{enumerate}
    \item $\varphi(0)=y$ and  $\varphi(1)=x$, and 
    \item $\mathrm{supp}(\varphi(t))=\mathrm{supp}(x)$ for all $t\in (0,1]$.
\end{enumerate}
Indeed, $\varphi$ can be chosen to be a line segment if  $x$'s $\ell$th coordinate is not the only coordinate of $x$ with absolute value equal to one and to consist of two line segments glued together otherwise. 

\begin{claim}
    $\mathrm{supp}(G(\varphi(t)))=\mathrm{supp}(G(x))$ for all $ t\in (0,1]$.  
\end{claim}

\begin{proof}
 Let \[A=\{t\in(0,1]\mid \mathrm{supp}(G(\varphi(t)))=\mathrm{supp}(G(x))\}.\] Since $A$ is non empty and $(0,1]$ is connected, it is enough to show that $A$ is a clopen subset of $(0,1]$. If $A$ is not open, then there is $t_0\in A$ and $(t_n)_{n=1}^\infty \subset (0,1]\setminus A$ converging to $t_0$, whereas, if $A$ is not closed, there is $t_0 \in (0,1]\setminus A$ and $(t_n)_{n=1}^\infty \subset A$ converging to $t_0$. In both situations
\begin{equation}\label{Eq.124124}\mathrm{supp}(G(\varphi(t_n)))\neq \mathrm{supp}(G(\varphi(t_0)))\ \text{ for all }n\in\N.
\end{equation}
Since $|\mathrm{supp}(\varphi(t_n))|=|\mathrm{supp}(x)|=s+1$ for all $n\ge 0$, it follows once again from Proposition \ref{invariance of domain} that  $|\mathrm{supp}(G(\varphi(t_n)))|= s+1$ for all $n\ge 0$. Then, \eqref{Eq.124124} allows us to pick, for each $n\in\N$, an    $i_n\in \mathrm{supp}(G(\varphi(t_0)))$ which is  not in  $\mathrm{supp}(G(\varphi(t_n)))$. By the pigeonhole principle, going to a subsequence if necessary, we can assume that the sequence $(i_n)_n$ is constant,  say $i_0=i_n$ for all $n\in\N$. Then, 
   \[0=\lim_{n}G_{i_0}(\varphi(t_n))=G_{i_0}(\varphi(t_0))\neq 0;\]
a contradiction. We have shown that $A$ is clopen in $(0,1]$, which concludes the proof of this claim. 
\end{proof}

The previous claim and  \eqref{EqGzeroatiforx} imply that
\[G_{i}(\varphi(t))=0\text{ for all } t\in (0,1].\]
On the other hand, as $i\in \mathrm{supp}(y)$, \eqref{Eq.suppy=suppGy} implies that $  G_{i}(y)\neq 0$. As $\varphi(0)=y$, this contradicts the continuity of $G_{i}\circ \varphi$ at $0$.
\end{proof}

\begin{proof}[Proof of Theorem \ref{Thm.Intro.Supp.Card.Preserv}]
Let $k \geq 19^{2d-1}+1$ and let $F: S_{\ell_\infty^k} \to S_{\ell_1^k}$ be a homeomorphism that does not increase support size. Applying Lemma \ref{LemmaCardOfSupportPreservingToSupportPreserving} to $F$, we may assume furthermore that the map $F$ is support preserving. The result then follows from Theorem \ref{Thm.Main}.
\end{proof}

\section{Step preserving  homeomorphisms}\label{SectionStepPreservation}

This section uses Lemma \ref{Lemma.Bounds.Max.Alpha.Beta(bis)}  in order to obtain Theorem \ref{Thm.Intro.Homeo.Equiv.Intro}. We adopt the notation and assumptions from Section~\ref{SectionMainLemma}. In particular, $(X,\|\cdot \|_X)$ is a Banach space with a normalized, 1-unconditional basis $(e_i)_{i=1}^\infty$ that is not equivalent to the standard $c_0$ basis, and each $X_k$ is the span of $(e_i)_{i=1}^k$.

\begin{theorem}\label{Thm.Intro.Homeo.Equiv(bis)} Let $d\in \N$ and let $P \subseteq \N$ and $M = \{k_j\}_{j=1}^\infty \subseteq \N$ be subsets satisfying {\bf Assumption} $\boldsymbol{A(d,\frac{1}{2})}$ defined in \eqref{AssumptionA}. Then for any  $k > k_{2d}$ and any step preserving continuous map   $F:S_{\ell_\infty^k}\to S_{X_k}$ such that $F(1,\ldots,1)\neq F(-1,\ldots,-1)$, we have $\omega_F(\frac1d)\ge \frac12$. In particular, there is no sequence $(F_k:S_{\ell_\infty^k}\to S_{X_k})_k$ of equi-uniformly continuous step  preserving homeomorphisms.
\end{theorem}

\begin{proof}
Let $ k > k_{2d}$. Then there exist elements $\bar m,\bar n\in [M]^d$ with $m_1<n_1<\ldots<m_d<n_d<k$. We fix such $\bar m$, $\bar n$ for the rest of the proof. Let $F:S_{\ell_\infty^k}\to S_{X_k}$ be a step preserving map such that $F(1,\ldots,1)\neq F(-1,\ldots,-1)$.  Set 
\begin{align*}
    y(\bar m)=\sum_{s=1}^{d}\Big(1-\frac{s-1}{d}\Big)1_{(m_{s-1},m_s]\cap P}+\sum_{s=1}^{d}\Big(-1+\frac{s-1}{d}\Big)1_{(m_{s-1},m_s]\cap P^c}.
\end{align*}
We see $y(\bar m)$  as an element in $S_{\ell_\infty^k}$. Visually, the graph of $y(\bar m)$ is a union of a staircase descending from $1$ to $0$ on the elements of $P$ and a staircase ascending from $-1$ to $0$ on the elements of $P^c$. Alternatively, seeing $y(\bar m)$ as an element in $S_{\ell_\infty^k}$,   we have  
   \[y(\bar m)=x(\bar m,v,k),\]
   where $v=((a_{s})_{s=1}^{{d}}, (b_s)_{s=1}^{{d}},c)$ is given by 
   \[a_s=1-\frac{s-1}{d},  \ b_s=-1+\frac{s-1}{d}, \text{ and } c=0,\]
   for all $s\in \{1,\ldots, {d}\}$.

Let now $\varphi_{\bar m,k}\colon [0,1]\to S_{\ell_\infty^k}$ be the piecewise affine map  given by joining two line  segments: the line segment joining  $(1,\ldots, 1)$ to $y(\bar m)$ for $t$ going from $0$ to $1/2$ and the line segment joining   $y(\bar m)$ to $(-1,\ldots, -1)$ for $t$ going from $1/2$ to $1$. Being the union of line segments connecting points in the same facet of $S_{\ell_\infty^k}$, the image of $\varphi_{\bar m, k}$ is in $S_{\ell_\infty^k}$. Moreover, notice that for each $t\in [0,1]$, the vector $\varphi_{\bar m,k}(t)$ is constant on the following subsets of $\{1,\ldots, k\}$:
\[(m_{s-1},m_s]\cap P,\   (m_{s-1},m_s]\cap P^c, \text{ and }   (m_d,k],\] 
for $s\in \{1,\ldots, d\}$.
Indeed, this is the case since this is true for $(1,\ldots,1)$, $y(\bar m)$, and $(-1,\ldots, -1)$, and since each $\varphi_{\bar m,k}(t)$ is a linear combination of them (it is in fact  a linear combination  of either the first two   vectors or  the last two vectors) depending on $t$.

We recall that $\psi(k)=\|\sum_{i=1}^ke_i\|$. As $F \colon S_{\ell_\infty^k}\to S_{X_k}$ is step preserving, 
\[F(1,\ldots, 1)=\pm \Big(\frac{1}{\psi(k)},\ldots, \frac{1}{\psi(k)}\Big)\ \text{ and }F(-1,\ldots, -1)=\pm\Big(-\frac{1}{\psi(k)},\ldots, -\frac{1}{\psi(k)}\Big).\] 
As $F(1,\ldots,1)\neq F(-1,\ldots,-1)$,  replacing $F$ by $-F$ if necessary, we can then assume that 
\[F(1,\ldots, 1)=\Big(\frac{1}{\psi(k)},\ldots, \frac{1}{\psi(k)}\Big)\ \text{ and }F(-1,\ldots, -1)=\Big(-\frac{1}{\psi(k)},\ldots, -\frac{1}{\psi(k)}\Big).\]
Since $F$ is continuous, 
\[t\in [0,1]\to F(\varphi_{\bar m ,k}(t))\in S_{X_k}\] 
is a continuous path connecting a point whose coordinates in $(m_d,k]$ are positive to a point whose coordinates in $(m_d,k]$ are negative. As remarked above, for any given $t \in [0,1]$, the value $F_{i}(\varphi_{\bar m ,k}(t))$ is the same among all coordinates $i \in (m_d,k]$.
Therefore, by the intermediate value theorem,   there is $t_{\bar m,k}\in [0,1]$ such that 
\begin{equation}\label{Eq.Const.u.0}
F_{i}(\varphi_{\bar m,k}(t_{\bar m,k}))=0\text{ for all }i\in (m_d, k].
\end{equation}
By the definition of $y(\bar m)$ and   $\varphi_{\bar m,k}$,   there are  $(a_{s}(\bar m, k))_{s=1}^{d}$, $(b_{s}(\bar m, k))_{s=1}^{d}\in [-1,1]^d$, and $c(\bar m, k)\in [-1,1]$  such that   
\begin{align}\label{Eq.varphi.x.m.u.k} 
\varphi_{\bar m, k} (t_{\bar m,k})=\sum_{s=1}^{d}a_s(\bar m, k)1_{(m_{s-1},m_s]\cap P}+\sum_{s=1}^{d}b_s(\bar m, k)1_{(m_{s-1},m_s]\cap P^c}+ c(\bar m, k)1_{(m_d,k]}.\notag
\end{align}
In other words, $\varphi_{\bar m, k} (t_{\bar m,k})=x(\bar m, u,k)$, for $u=\left((a_s(\bar m,k))_{s=1}^d,(b_s(\bar m,k))_{s=1}^d, c(\bar m, k)\right)$. By \eqref{Eq.Const.u.0}, we have that 
\[F_{i}(x(\bar m, u,k))=0\text{ for all }i\in (m_d,k].\]
Therefore, $\bar m,\bar n, k,$ and $F$ satisfy the required properties in the statement of Lemma \ref{Lemma.Separation}, and this lemma implies that 
\begin{equation}\label{Eq.Fk.big}\|F_k(x(\bar m,u,k))-F_k(x(\bar n,u,k))\|_X\geq \frac12.
\end{equation}

Note that $\|y(\bar m)-y(\bar n)\|_\infty \le \frac1d$ by \eqref{eq:d(interlacing)}. Moreover, there exists $\lambda \in [0,1]$ such that either $x(\bar m,u,k)=\lambda y(\bar m)+(1-\lambda)(1,\ldots,1)$ and $x(\bar n,u,k)=\lambda y(\bar n)+(1-\lambda)(1,\ldots,1)$ or $x(\bar m,u,k)=\lambda y(\bar m)+(1-\lambda)(-1,\ldots,-1)$ and $x(\bar n,u,k)=\lambda y(\bar n)+(1-\lambda)(-1,\ldots,-1)$. Then, it easily follows from the triangle inequality that $\| x(\bar m,u,k)- x(\bar n,u,k)\|_\infty \le \frac1d$. By \eqref{Eq.Fk.big}, we have then shown that 
$$\omega_F\Big(\frac1d\Big)\ge \frac12.$$
\end{proof}

With Theorem~\ref{Thm.Intro.Homeo.Equiv(bis)} in hand, we can deduce the following stronger version of Theorem~\ref{Thm.Intro.Homeo.Equiv.Intro} from the introduction.

\begin{theorem}\label{Thm.Intro.Homeo.Equiv}
Let $d \in \N$   and $1\leq r<\infty$. Then for any  $k \geq \lceil 19^r\rceil^{2d-1}+1$ and any step preserving continuous map   $F:S_{\ell_\infty^k}\to S_{\ell_r^k}$ such that $F(1,\ldots,1)\neq F(-1,\ldots,-1)$, we have $\omega_F(\frac1d)\ge \frac12$. In particular, there is no sequence $(F_k:S_{\ell_\infty^k}\to S_{\ell_r^k})_{k=1}^\infty$ of equi-uniformly continuous step preserving homeomorphisms.
\end{theorem}

\begin{proof} 
The theorem follows immediately from Theorem~\ref{Thm.Intro.Homeo.Equiv(bis)} and, by Remark~\ref{rmk:AssumptionAl1}, that  $\ell_r$ with its standard basis admits sets $P \subseteq \N$ and $M = \{k_j\}_{j=1}^\infty \subseteq \N$ satisfying {\bf Assumption} $\boldsymbol{A(d,\frac{1}{2})}$ with  $k_{2d} < \lceil 19^r\rceil^{2d-1}$.
\end{proof}

\begin{proof}
    [Proof of Theorem \ref{Thm.Intro.Homeo.Equiv.Intro}]
This is simply Theorem \ref{Thm.Intro.Homeo.Equiv} with $r=1$.
\end{proof}

We conclude this section with a nonlinear characterization of the standard $c_0$ basis among unconditional bases. 
\begin{cor}
Let $X$ be Banach space with a normalized unconditional basis $(e_i)_i$ and for for each $k\in\N$ let $X_k$ be the linear span of $\{e_1,\ldots,e_k\}$. Then $(e_i)_i$ is equivalent to the canonical basis of $c_0$ if and only if there exists a sequence $(F_k:S_{\ell_\infty^k} \to S_{X_k})_k$ of equi-uniformly continuous step  preserving homeomorphisms.
\end{cor}

\begin{proof} Assume first that  $(e_i)_i$ is equivalent to the $c_0$ basis and let $F_k:S_{\ell_\infty^k} \to S_{X_k}$ be defined by 
\[F_k(a_1,\ldots,a_k)=\frac{\sum_{i=1}^k a_i e_i}{\|\sum_{i=1}^k a_i e_i\|_X},\ \ (a_1,\ldots,a_k) \in S_{\ell_\infty^k}.\]
Then $(F_k)_k$ is a sequence of equi-bi-Lipschitz step preserving maps from $S_{\ell_\infty^k}$ onto $S_{X_k}$.

Assume now that $(e_i)_i$ is not equivalent to the $c_0$ basis. After renorming, we may assume that $(e_i)_i$ is $1$-unconditional. Then the conclusion follows from Theorem \ref{Thm.Intro.Homeo.Equiv(bis)}. 
\end{proof}

\section{A concentration inequality}\label{SectionConcentration}

Before delving into the proof of Theorem \ref{ThmConcentration.LastSection}, let us relate this theorem to Kalton's Property $\mathcal Q$. Recalling the notation in Theorem \ref{ThmConcentration.LastSection}, as in  Section \ref{Eq.Def.Z}, we write  
\[z(\bar m)=\sum_{s=1}^d\left(1-\frac{s-1}{d}\right)1_{(m_{s-1},m_s]},\]
when $d\in \N$ and $\bar m\in [\N]^d$. As long as $k\geq m_d$, we view $z(\bar m)$ as an element of norm $1$ in $\ell_\infty^k$ and, consequently, in $c_0$. While we refrain from properly defining Property $\mathcal Q$ --- its definition depends on {\it Kalton's interlaced graphs} ---, we recall that $\ell_1$ has this property (see \cite[Theorem 4.1]{Kalton2007QJM}), and an immediate consequence of this fact is that if $F\colon S_{c_0}\to S_{\ell_1}$ is a uniformly continuous function, then for all $\eps>0$  there exists $d_0 \in \N$ such that for all $d\ge d_0$ there is an infinite $M\subseteq \N$ such that 
\[\left\|F(z(\bar m))-F(z(\bar n))\right\|\leq \eps\]
for all $\bar m,\bar n\in [M]^d$.
    
We again adopt the notation and assumptions from Section~\ref{SectionMainLemma}. In particular, $(X,\|\cdot\|_X)$ is a Banach space with a normalized, 1-unconditional basis $(e_i)_{i=1}^\infty$ that is not equivalent to the standard $c_0$ basis, each $X_k$ is the span of $(e_i)_{i=1}^k$, and $\psi(k) = \|1_{[1,k]}\|_X$. In this section, we shall also look at the positive part of the spheres of $X_k$: for each $k\in\N$, we let
\[S^+_{X_k}=\Big\{\sum_{i=1}^ka_ie_i\in S_{X_k}\mid a_1,\ldots, a_k\geq 0\Big\}.\]

\begin{theorem}\label{ThmConcentration.LastSection(bis)} 
Fix $d\in \N$ and $0<  \eps \leq 1 $. Let $P \subseteq \N$ and $M = \{k_j\}_{j=1}^\infty \subseteq \N$ be subsets satisfying {\bf Assumption} $\boldsymbol{A(d,\eps/4})$ from \eqref{AssumptionA}. Then for all $\bar m \in Q = [\{k_{2j}\}_{j=1}^\infty]^d$, all $k\in Q$ with $k > m_d$, and all step preserving maps $F:S_{\ell_\infty^k}^+ \to S_{X_k}^+$ with $\omega_F(\frac{1}{d})\le \frac{\eps}{8}$, we have 
\[\Big\|F(z(\bar m))-\frac{1}{\psi(k)}\sum_{i=1}^k e_i\Big\|_X\le \eps.\]
\end{theorem}

\begin{proof}
Choose $\bar n \in [M]^d$ with $m_1<n_1<\cdots<m_d<n_d < k$. We can find such an element $\bar n$ from $[\{k_{2j+1}\}_{j=1}^\infty]^d \subseteq [M]^d$. Let $F:S_{\ell_\infty^k}^+ \to S_{X_k}^+$ be a step preserving map with $\omega_F(\frac1d)\le \frac{\eps}{8}$. Notice that our data $\bar m, \bar n, k, F$ satisfy the hypotheses of Lemma \ref{Lemma.Separation}, except possibly for the last one: it does not need to be the case that $F_{i}(z(\bar m))=0$ for all $i\in (m_d,k]$. For $\bar p \in \{\bar m,\bar n\}$ and $s\in \{1,\ldots, d\}$, let 
 \[\alpha_s(\bar p,k)=F_{i}(z(\bar p)),\ \ \text{for } p_{s-1}<i\le p_s.\]
Then we claim that 
\begin{equation}\label{Eq.??}
\Big\|\sum_{s=1}^d \alpha_s(\bar m,k)1_{(m_{s-1},m_s]}\Big\|_X \le \frac{\eps}{4 }.
\end{equation}
Indeed, since $(e_i)_i$ is $1$-unconditional 
\[ \Big\|\sum_{s=1}^d \alpha_s(\bar m,k)1_{(m_{s-1},m_s]\cap P^c}\Big\|_X \le 1\ \  \text{and}\ \  \Big\|\sum_{s=1}^d\alpha_s(\bar n,k)1_{(n_{s-1},n_s]\cap P}\Big\|_X \le 1.\] 
It follows from Lemma~\ref{Lemma.Bounds.Max.Alpha.Beta(bis)} that 
\begin{equation*}\label{Eq.t.l.d}
\Big\|\sum_{s=1}^d\alpha_s({\bar m},k)1_{(m_{s-1},n_{s-1}]}\Big\|_X\le \frac{\eps}{16}  \text{  and } \Big\|\sum_{s=1}^d\alpha_s({\bar n},k)1_{(n_{s-1},m_{s}]}\Big\|_X\le \frac{\eps}{16}.
 \end{equation*}
Therefore, as $\|z(\bar n)-z(\bar m)\|_\infty\leq \frac1d$ by \eqref{eq:d(interlacing)}, we get that 
\begin{align*}
\Big\|\sum_{s=1}^d &\alpha_s(\bar m,k)1_{(m_{s-1},m_s]}\Big\|_X \\
&\le \Big\|\sum_{s=1}^d (\alpha_s(\bar m,k)-\alpha_s(\bar n,k))1_{(n_{s-1},m_s]}\Big\|_X+\Big\|\sum_{s=1}^d \alpha_s(\bar m,k)1_{(m_{s-1},n_{s-1}]}\Big\|_X\\ 
& + \Big\|\sum_{s=1}^d \alpha_s(\bar n,k)1_{(n_{s-1},m_s]}\Big\|_X \\
&\le \|F(z(\bar m))-F(z(\bar n))\|_X+\frac{\varepsilon}{8}\le \omega_F\Big(\frac1d\Big)+\frac{\eps}{8}\le \frac{\eps}{4}.
\end{align*}
This proves \eqref{Eq.??}.

Let us write
\[F(z(\bar m))=\sum_{s=1}^d \alpha_s(\bar m,k)1_{(m_{s-1},m_s]} + \gamma (\bar m,k)1_{(m_d,k]}.\]
By \eqref{Eq.??}, and since $\|F(z(\bar m))\|_X=1$ and $(e_i)_i$ is 1-unconditional, we get that  
\[\Big(1-\frac{\eps}{4}\Big)\le \gamma (\bar m,k)\|1_{(m_d,k]}\|_X\le 1.\]
  {\bf Assumption} $\boldsymbol{A(d, \eps/4})$ from \eqref{AssumptionA} yields that $\psi(q') \geq \frac{32}{\eps}\psi(q)$ for all $q,q'\in M$ with $q' > q$. So,
\[\Big(1-\frac{\eps}{32}\Big)\psi(k) \le \|1_{(m_d,k]}\|_X \le \psi(k).\]
Thus
\[\Big(1-\frac{\eps}{4}\Big)\frac{1}{\psi(k)} \le \gamma(\bar m,k) \le \Big(1-\frac{\eps}{32}\Big)^{-1}\frac{1}{\psi(k)}\le \Big(1+\frac{\eps}{4}\Big)\frac{1}{\psi(k)},\]
and 
\[\Big|\gamma(\bar m,k)-\frac{1}{\psi(k)}\Big|\le \frac{\eps}{4}\frac{1}{\psi(k)}.\]
Note also that, since $m_d<k \in Q$,  
\[\|\gamma(\bar m,k)1_{(0,m_d]}\|_X\le \Big(1+\frac{\eps}{4}\Big)\frac{\psi(m_d)}{\psi(k)}\le \Big(1+\frac{\eps}{4}\Big)\frac{\eps}{32}\le \frac{\eps}{16}.\]
Therefore,
\begin{align*}
\Big\|F(z(\bar m))-\frac{1}{\psi(k)}1_{(0,k]}\Big\|_X
\le & \Big\|\gamma(\bar m,k)1_{(m_d,k]}-\frac{1}{\psi(k)}1_{(0,k]}\Big\|_X\\
&+ \Big\|\sum_{s=1}^d \alpha_s(\bar m,k)1_{(m_{s-1},m_s]}\Big\|_X \\
\le &\|\gamma(\bar m,k)1_{(0,m_d]}\|_X+\Big|\gamma(\bar m,k)-\frac{1}{\psi(k)}\Big|\psi(k)+\frac{\eps}{4} \\
\le &\frac{\eps}{16}+\frac{\eps}{4}+\frac{\eps}{4} \\
<& \eps.
\end{align*}
\end{proof}

We can now apply Theorem~\ref{ThmConcentration.LastSection(bis)} to prove a version of Theorem~\ref{ThmConcentration.LastSection.Intro} for $\ell_r$ in general.

\begin{theorem}\label{ThmConcentration.LastSection}
Fix $1\leq r<\infty$, $d\in \N$ and $\eps > 0$. Let $Q = \{k_{2j}\}_{j=1}^\infty \subseteq \N$, where $k_j = a^{j-1}$ and $a := \lceil (32/\eps + 2)^r \rceil$. Then for all $\bar m\in [Q]^d$, all $k\in Q$ with $k>m_d$,   and all step preserving maps $F:S_{\ell_\infty^k}^+ \to S_{\ell_r^k}^+$ with $\omega_F(\frac{1}{d})\le \frac{\eps}{8}$, we have
\[\Big\|F(z(\bar m))-\frac{1}{k^{1/r}}\sum_{i=1}^k e_i\Big\|_r\le \eps,\]
where $(e_i)_{i=1}^k$ denotes the standard basis of $\ell_r^k$. 
\end{theorem}

\begin{proof} 
By Remark \ref{rmk:AssumptionAl1}, the subsets $P = 2\N$ and $M = \{k_j\}_{j=1}^\infty \subseteq \N$ satisfy {\bf Assumption} $\boldsymbol{A(d, \eps/4})$ \eqref{AssumptionA}, where  $k_j := a^{j-1}$ and  $a := \lceil (32/\eps + 2)^r \rceil$. Thus, hypotheses of Theorem~\ref{ThmConcentration.LastSection(bis)} are satisfied.
\end{proof} 

\begin{proof}
    [Proof of Theorem \ref{ThmConcentration.LastSection.Intro}]
    This is simply Theorem \ref{ThmConcentration.LastSection} with $r=1$.
\end{proof}

\medskip We finish this paper introducing a local version of Kalton's Property $\mathrm Q$ and showing that it characterizes when unconditional bases are equivalent to the standard basis of $c_0$.

\begin{definition}\label{Defi.Local.Prop.Q} Let $(e_i)_i$ be an unconditional basis of a Banach space $X$. We say that $(e_i)_i$ has the {\it local Property $\mathcal Q$} if there exists a constant $\gamma>0$ such that for any $d\in \N$, $\eps \in (0,1)$, there exists an infinite subset $Q \subseteq \N$ such that for all $\bar m \in [Q]^d$, all $k\in Q$ with $m_d<k$, and for all $F:S_{\ell_\infty^k}^+ \to S_{X_k}^+$ step preserving map with $\omega_F(\frac1d)\le \gamma\eps$ , we have 
\[\Big\|F(z(\bar m))-\frac{\sum_{i=1}^ke_i}{\|\sum_{i=1}^ke_i\|_X}\Big\|_X\le \eps.\]
\end{definition}

We can now characterize the standard $c_0$ basis in terms of the local property $\mathcal Q$. 
\begin{cor}\label{Corollary.Charactc0LocalPropQ} Let $(e_i)_i$ be an unconditional basis of a Banach space $(X,\|\ \|_X)$. Then $(e_i)_i$ has the local Property $\mathcal Q$ if and only if it is not equivalent to the standard $c_0$ basis.
\end{cor}

\begin{proof}
Assume first that $(e_i)_i$ is equivalent to the standard $c_0$ basis. Then, $F_k:S_{\ell_\infty^k}^+ \to S_{X_k}^+$  such that 
\[F_k(a_1,\ldots,a_k)=\frac{\sum_{i=1}^k a_i e_i}{\|\sum_{i=1}^k a_i e_i\|_X},\ \ (a_1,\ldots,a_k) \in S_{\ell_\infty^k}^+,\]
defines a sequence of equi-bi-Lipschitz equivalences. This, together with the fact that for $\bar m \in [\N]^d$ and $k>m_d$, $\|z(\bar m)-\sum_{i=1}^ke_i\|_\infty=1$, shows that $(e_i)_i$  fails to have the local Property $\mathcal Q$.

Assume now that $(e_i)_i$ is an unconditional basis of $(X,\|\cdot \|_X)$, which is not equivalent to the $c_0$ basis. Let $|\cdot |$ be an equivalent norm on $X$ for which $(e_i)_i$ is 1-unconditional. Then, for any $k\in \N$, $x\mapsto \frac{x}{\|x\|}$ is a step preserving bi-Lipschitz equivalence from $S_{(X_k,|\cdot |)}^+$ onto $S_{(X_k,\|\cdot \|_X)}^+$ mapping $\frac{\sum_{i=1}^ke_i}{|\sum_{i=1}^ke_i|}$ to $\frac{\sum_{i=1}^ke_i}{\|\sum_{i=1}^ke_i\|_X}$. Thus, the conclusion follows from Theorem \ref{ThmConcentration.LastSection(bis)}.
\end{proof}

\begin{acknowledgements*} This paper was written under the auspices of the
American Institute of Mathematics (AIM) SQuaREs program as part of the
``Nonlinear geometry of Banach spaces''  SQuaRE project. B.\ M.\ Braga is thankful to Rufus Willett for introducing him to Property (H) and to Quico for teaching him that spheres do not need to be rounded.
\end{acknowledgements*}

\end{document}